\documentclass[12]{amsart}
\usepackage{amsmath,amssymb,amsthm,color,enumerate,comment,centernot,enumitem,url,cite}
\usepackage{graphicx,relsize,bm}
\usepackage{mathtools}
\usepackage{array}

\makeatletter
\newcommand{\tpmod}[1]{{\@displayfalse\pmod{#1}}}
\makeatother

\newtheorem{thm}{Theorem}[section]
\newtheorem{lemma}[thm]{Lemma}

\theoremstyle{remark}

    \newtheorem*{rems}{{\bf Remarks}}

\theoremstyle{definition}

\newtheorem{rem}[thm]{Remark}

\theoremstyle{THM}

\newcommand{\abs}[1]{\left|{#1}\right|}

\def\FF {{\mathcal F}}

\def\Z {{\mathbb Z}}

\def\QQ {{\mathcal Q}}

\def\Q {{\mathbb Q}}

\def\D {{\mathcal D}}

\def\F {{\mathbb F}}
\def\D {{\mathcal D}}
\def\Z {{\mathbb Z}}
\def\Q {{\mathbb Q}}

\def\Gal{{\mbox{{\rm{Gal}}}}}

\makeatletter
\@namedef{subjclassname@2020}{%
  \textup{2020} Mathematics Subject Classification}
\makeatother

\def\red#1 {\textcolor{red}{#1 }}
\def\blue#1 {\textcolor{blue}{#1 }}

\numberwithin{equation}{section}

\def\Z {{\mathbb Z}}

\newcommand{\Mod}[1]{\ (\mathrm{mod}\enspace #1)}
\newcommand{\mmod}[1]{\ \mathrm{mod}\enspace #1}
\begin{document}

\title[Monogenic Even Cyclic Sextic Polynomials]{Monogenic Even Cyclic Sextic Polynomials}


\author{Lenny Jones}
\address{Professor Emeritus, Department of Mathematics, Shippensburg University, Shippensburg, Pennsylvania 17257, USA}
\email[Lenny~Jones]{doctorlennyjones@gmail.com}

\date{\today}

\begin{abstract}
Suppose that $f(x)\in \Z[x]$ is monic and irreducible over $\Q$ of degree $N$. We say that $f(x)$ is \emph{monogenic} if $\{1,\theta,\theta^2,\ldots ,\theta^{N-1}\}$ is a basis for the ring of integers of $\Q(\theta)$, where $f(\theta)=0$, and we say $f(x)$ is \emph{cyclic} if the Galois group of $f(x)$ over $\Q$ is isomorphic to the cyclic group of order $N$. In this note, we prove that there do not exist any monogenic even cyclic sextic binomials or trinomials. Although the complete story on monogenic even cyclic sextic quadrinomials remains somewhat of a mystery, we nevertheless determine that the union of three particular infinite sets of cyclic sextic quadrinomials contains exactly four quadrinomials that are monogenic with distinct splitting fields. We also show that the situation can be quite different for quadrinomials whose Galois group is not cyclic.
\end{abstract}


\maketitle
\section{Introduction}\label{Section:Intro}
Suppose that $f(x)\in \Z[x]$ is monic and irreducible over $\Q$, with $\deg(f)=N$. We say that $f(x)$ is \emph{monogenic} if $\{1,\theta,\theta^2,\ldots ,\theta^{N-1}\}$ is a basis for the ring of integers of $\Q(\theta)$, where $f(\theta)=0$, and we say $f(x)$ is \emph{cyclic} if the Galois group of $f(x)$ over $\Q$, denoted here as $\Gal(f)$, is isomorphic to $C_N$, the cyclic group of order $N$. It was shown recently that there exist only three distinct monogenic even cyclic quartic trinomials \cite{JonesBAMSEQT}. In this note, we investigate the existence of monogenic cyclic even sextic polynomials. We prove that there do not exist any monogenic even cyclic sextic binomials or trinomials by proving that no even cyclic binomials or trinomials exist. Although the complete story on monogenic even cyclic sextic quadrinomials remains somewhat of a mystery, we nevertheless determine that the union of three particular infinite sets of cyclic sextic quadrinomials contains exactly four quadrinomials that are monogenic with distinct splitting fields. 
More precisely, our main theorem is:
\begin{thm}\label{Thm:Main} Let $a,b,c\in \Z$.
\begin{enumerate}
  \item \label{Main I:1} If $f(x)=x^6+ax^4+bx^2+c$ is irreducible over $\Q$ with $ab=0$, then $\Gal(f)\not \simeq C_6$. Consequently, there exist no monogenic even cyclic sextic binomials or trinomials.
  \item \label{Main I:2} Let
  \[f(x)=x^2(x^2+a)^2+b=x^6+2ax^4+a^2x^2+b\] be irreducible over $\Q$. Then,
  $\Gal(f)\simeq C_6$ if and only if $4a^3b-27b^2$ is a square. Moreover, if $f(x)\in \FF_1$, where $\FF_1$ is the infinite set
    \begin{equation}\label{F1}
    \left\{f(x) :\ \mbox{$f(x)$ is irreducible over $\Q$ and $4a^3b-27b^2$ is a square.}\right\},
    \end{equation} then $f(x)$ is monogenic if and only if
    \[f(x)\in \QQ_1:=\{x^6-6x^4+9x^2-3, \ x^6+6x^4+9x^2+1\}.\] 
  \item \label{Main I:3} Let
  \[f(x)=x^6+a(bx^2+1)^2=x^6+ab^2x^4+2abx^2+a\] be irreducible over $\Q$. Then,
  $\Gal(f)\simeq C_6$ if and only if $4ab^3-27$ is a square. Moreover, if $f(x)\in \FF_2$, where $\FF_2$ is the infinite set
    \begin{equation}\label{F2}
    \left\{f(x) :\ \mbox{$f(x)$ is irreducible over $\Q$ and $4ab^3-27$ is a square.}\right\},
    \end{equation} then $f(x)$ is monogenic if and only if
    \[f(x)\in \QQ_2:=\{x^6-7x^4+14x^2-7, \ x^6+9x^4+6x^2+1\}.\]
   \item \label{Main I:4} For any integer $n$, define $\FF_3:=\{f_n(x): n\in \Z\}$, where
   \[f_n(x):=x^6+(n^2+5)x^4+(n^2+2n+6)x^2+1.\] Then $f_n(x)$ is irreducible over $\Q$ and $\Gal(f_n)\simeq C_6$. Moreover, $f_n(x)$ is monogenic if and only if
   \[f_n(x)\in \QQ_3=\{f_{-2}(x),\ f_{-1}(x), \ f_0(x),\ f_1(x)\}.\] 
   \item \label{Main I:5} Furthermore, there are exactly four distinct quadrinomials in $\QQ_1\cup \QQ_2\cup \QQ_3$, and they are:
  \begin{gather*}
  x^6-6x^4+9x^2-3,\quad x^6+6x^4+9x^2+1,\\
  x^6-7x^4+14x^2-7, \quad x^6+5x^4+6x^2+1.
  \end{gather*} Here the term ``distinct" means that the qudarinomials have different splitting fields. 
\end{enumerate}
\end{thm}
\begin{rems}\label{Rem:distinct}
As far as we can determine, the $C_6$-families $\FF_1$, $\FF_2$ and $\FF_3$ given in Theorem \ref{Thm:Main} are new and do not appear in the current literature.
\end{rems}

Several authors have contributed to the literature concerning the Galois groups of sextic polynomials \cite{AJ,AL,BS,HJMS,S1,S2}. In \cite{AJ}, Awtrey and Jakes gave a classification of the Galois groups of even sextic polynomials, while in \cite{AL}, the focus was on reciprocal polynomials. A somewhat more general analysis was conducted in \cite{S1,S2}, which predated the work in \cite{AJ}. Although the focus in \cite{HJMS} was on power-compositional sextics, there is no overlap with the current paper. A result along the same lines as Theorem \ref{Thm:Main} was given by Bremner and Spearman in \cite{BS}, where it was shown that, up to scaling, $x^6+133x+209$ is the only cyclic sextic of the form $x^6+Ax+B$.
As far as incorporating the notion of monogenicity into the mix, 
the only example we could find is \cite{LSW}, where the authors focus on the relationship between the monogenicity of a sextic field generated by $f(x)$ with $\abs{f(0)}=1$ and the monogenicity of a cubic subfield.

In section \ref{Sec:noncyclic}, we show that the situation when $\Gal(f)\not \simeq C_6$ can be quite different than what occurs in Theorem \ref{Thm:Main}. In particular, we construct an infinite collection of non-distinct monogenic even sextic quadrinomials that have Galois group isomorphic to the alternating group $A_4$. 

\section{Preliminaries}\label{Section:Prelim}
Suppose that $f(x)\in \Z[x]$ is monic and irreducible over $\Q$. Let $K=\Q(\theta)$, where $f(\theta)=0$, and let $\Z_K$ denote the ring of integers of $K$. Then, we have \cite{Cohen}
\begin{equation} \label{Eq:Dis-Dis}
\Delta(f)=\left[\Z_K:\Z[\theta]\right]^2\Delta(K),
\end{equation}
where $\Delta(f)$ and $\Delta(K)$ denote the discriminants over $\Q$, respectively, of $f(x)$ and the number field $K$.
Thus, we have the following theorem from \eqref{Eq:Dis-Dis}.
\begin{thm}\label{Thm:mono}
The polynomial $f(x)$ is monogenic if and only if
  $\Delta(f)=\Delta(K)$, or equivalently, $\Z_K=\Z[\theta]$.
\end{thm}


The first lemma, which follows from \cite{AJ}, will be useful in the proof of Theorem \ref{Thm:Main} and also in the proof of Theorem \ref{Thm:A4} in Section \ref{Sec:noncyclic}.
\begin{lemma}\label{Lem:AJ}
 Let $f(x)=x^6+ax^4+bx^2+c\in \Z[x]$ be irreducible over $\Q$. Let $g(x)=x^3+ax^2+bx+c$ and $h(x)=x^6-bx^4+acx^2-c^2$. Then
  \begin{enumerate}
  \item \label{C6} $\Gal(f)\simeq C_6$ if and only if  all of the following conditions hold:
 \begin{enumerate}
 \item $-c$ is not a square in $\Z$,
 \item $\Delta(g)$is a square in $\Z$, 
  \item $h(x)$ is reducible over $\Q$.
  \end{enumerate}
  \item \label{A4} $\Gal(f)\simeq A_4$ if and only if  all of the following conditions hold:
 \begin{enumerate}
 \item $-c$ is a square in $\Z$,
 \item $\Delta(g)$ is a square in $\Z$, 
  \item $h(x)$ is irreducible over $\Q$.
  \end{enumerate}
  \end{enumerate}
  \end{lemma}

The next lemma follows from \cite[Theorem 3.3]{HJMS}.
\begin{lemma}\label{Lem:HJMS} Let $A,B,C\in \Z$ and let $G(x)=x^3+Ax^2+Bx-C^2$ be irreducible over $\Q$.  If
\[F(x):=G\left(x^2\right)=x^6+Ax^4+Bx^2-C^2\] is reducible over $\Q$, then
 \begin{equation}\label{Eq:HJMAconditions}
 A=2n-m^2 \quad \mbox{and} \quad B=n^2-2mc
 \end{equation}
 for some $m,n\in \Z$.
\end{lemma}

The following theorem, known as \emph{Dedekind's Index Criterion}, or simply \emph{Dedekind's Criterion} if the context is clear, is a standard tool used in determining the monogenicity of a polynomial.
\begin{thm}[Dedekind \cite{Cohen}]\label{Thm:Dedekind}
Let $K=\Q(\theta)$ be a number field, $T(x)\in \Z[x]$ the monic minimal polynomial of $\theta$, and $\Z_K$ the ring of integers of $K$. Let $p$ be a prime number and let $\overline{ * }$ denote reduction of $*$ modulo $p$ (in $\Z$, $\Z[x]$ or $\Z[\theta]$). Let
\[\overline{T}(x)=\prod_{i=1}^k\overline{\tau_i}(x)^{e_i}\]
be the factorization of $T(x)$ modulo $p$ in $\F_p[x]$, and set
\[h_1(x)=\prod_{i=1}^k\tau_i(x),\]
where the $\tau_i(x)\in \Z[x]$ are arbitrary monic lifts of the $\overline{\tau_i}(x)$. Let $h_2(x)\in \Z[x]$ be a monic lift of $\overline{T}(x)/\overline{h_1}(x)$ and set
\[F(x)=\dfrac{h_1(x)h_2(x)-T(x)}{p}\in \Z[x].\]
Then
\[\left[\Z_K:\Z[\theta]\right]\not \equiv 0 \pmod{p} \Longleftrightarrow \gcd\left(\overline{F},\overline{h_1},\overline{h_2}\right)=1 \mbox{ in } \F_p[x].\]
\end{thm}
The next two theorems are algorithmic versions of Theorem \ref{Thm:Dedekind} crafted for specific types of polynomials. 
The first of these theorems is an adaptation of a theorem due to Jakhar, Laishram and Yadav \cite[Theorem 1.2]{JLY} to the specific situation of item \eqref{Main I:2} in Theorem \ref{Thm:Main}.
\begin{thm}\label{Thm:JLY}
Let $a,b\in \Z$ and let
\[f(x)=x^2(x^2+a)^2+b=x^6+2ax^4+a^2x^2+b\] be irreducible over $\Q$. Then
\begin{equation*}
\Delta(f)=-2^6b^3(4a^3-27b)^2.
\end{equation*} Suppose that $K=\Q(\theta)$ where $f(\theta)=0$, and let $\Z_K$ denote the ring of integers of $K$. A prime divisor $p$ of $\Delta(f)$ does not divide the index $[\Z_K:\Z[\theta]]$ if and only if $p$ satisfies one of the following conditions:
\begin{enumerate}
  \item \label{JLY:1} when $p\mid b$, then $p^2\nmid b$;
  \item \label{JLY:2} when $p\mid a$ and $p\nmid b$ with $j\ge 1$ as the highest power of $p$ dividing 6, then
      \[\mbox{either} \ p\mid b_1 \ {and} \ p\nmid c_1 \quad {or} \quad  p\nmid (b_1\left((b_1b)^3+bc_1^3\right),\]
      where
      \[b_1=\frac{2a}{p} \quad  \mbox{and} \quad c_1=\frac{b+(-b)^{p^j}}{p};\]
  \item \label{JLY:3} when $2\nmid ab$, then $b\equiv 1 \pmod{4}$;
  \item \label{JLY:4} when $p\nmid 2ab$, then $p^2\nmid (4a^3-27b)$.
\end{enumerate}
\end{thm}

The next theorem is an adaptation of a theorem due to Jakhar, Kalwaniya and Kotyada \cite[Theorem 1.2]{JKK} to the specific situation of item \eqref{Main I:3} in Theorem \ref{Thm:Main}.
\begin{thm}\label{Thm:JKK}
Let $a,b\in \Z$ and let
\[f(x)=x^6+a(bx^2+1)^2=x^6+ab^2x^4+2abx^2+a\] be irreducible over $\Q$. Then
\begin{equation*}
\Delta(f)=-2^6a^5(4ab^3-27)^2.
\end{equation*} Suppose that $K=\Q(\theta)$ where $f(\theta)=0$, and let $\Z_K$ denote the ring of integers of $K$. A prime divisor $p$ of $\Delta(f)$ does not divide the index $[\Z_K:\Z[\theta]]$ if and only if $p$ satisfies one of the following conditions:
\begin{enumerate}
  \item \label{JKK:1} when $p\mid a$, then $p^2\nmid a$;
  \item \label{JKK:2} when $p\nmid a$ and $p\mid b$ with $j\ge 1$ as the highest power of $p$ dividing 6, then
      \[\mbox{either} \ p\mid b_1 \ {and} \ p\nmid c_1 \quad {or} \quad  p\nmid b_1\left(-c_1^3+ab_1^3\right),\]
      where
      \[b_1=\frac{2ab}{p} \quad  \mbox{and} \quad c_1=\frac{a+(-a)^{p^j}}{p};\]
  \item \label{JKK:3} when $2\nmid ab$, then $a\equiv 1 \pmod{4}$;
  \item \label{JKK:4} when $p\nmid 2ab$, then $p^2\nmid (4ab^3-27)$.
\end{enumerate}
\end{thm}
\section{The Proof of Theorem \ref{Thm:Main}}\label{Section:MainProof}
\begin{proof}
  To establish item \eqref{Main I:1}, we consider the three possible forms for $f(x)$, depending on whether both $a$ and $b$ are zero, or just one of $a$ and $b$ is zero. For each of these possibilities for $f(x)$, we let $g(x)$ and $h(x)$ be as defined in Lemma \ref{Lem:AJ}.

  Suppose first that $a=0$ and $b=0$, so that $f(x)=x^6+c$ and $g(x)=x^3+c$. Then $\Delta(g)=-27c^2$ is clearly not a square, and therefore, $\Gal(f)\not \simeq C_6$ by item \eqref{C6} of Lemma \ref{Lem:AJ}.

  Suppose next that $a=0$ and $b\ne 0$, so that
  \begin{equation}\label{Eq:fg}
  f(x)=x^6+bx^2+c \quad \mbox{and}\quad g(x)=x^3+bx+c.
   \end{equation} Assume, by way of contradiction, that $\Gal(f)\simeq C_6$. Then, by item \eqref{C6} of Lemma \ref{Lem:AJ}, we have that
   $h(x)=x^6-bx^4-c^2$ is reducible over $\Q$. Define
  \begin{equation}\label{Eq:h1}
  h_1(x):=x^3-bx^2-c^2.
   \end{equation} If $h_1(x)$ is irreducible over $\Q$, then since $h(x)$ is reducible, it follows from \eqref{Eq:HJMAconditions} of Lemma \ref{Lem:HJMS}, with \[G(x)=h_1(x),\quad F(x)=h(x),\quad A=-b, \quad B=0\quad  \mbox{and} \quad C=c\] that
   \begin{equation}\label{Eq:HJMSconditions1}
  -b=2n-m^2 \quad \mbox{and} \quad n^2=2mc.
  \end{equation}
  Combining the two equations in \eqref{Eq:HJMSconditions1}, we get that
  \begin{equation}\label{Eq:b1}
  b=\frac{n^4-8c^2n}{4c^2}.
  \end{equation} Then, by \eqref{Eq:b1}, we have that
  \[\Delta(g)=-(4b^3+27c^2)=\frac{-(n^6-20n^3c^2+108c^4)(n^3-2c^2)^2}{16c^6}.\]
   Observe that $2\mid n$ from \eqref{Eq:HJMSconditions1}. Hence,
  \[-(n^6-20n^3c^2+108c^4)\equiv 20c^4 \pmod{32}.\]
  Since 20 is not a square modulo 32, it follows that $\Delta(g)$ is not a square, which contradicts item \eqref{C6} of Lemma \ref{Lem:AJ}.
  Therefore, $h_1(x)$ must be reducible. Let $r\in \Z$ be a zero of $h_1(x)$. Then $r^3-br^2-c^2=0$ from \eqref{Eq:h1}, so that $b=r-(c/r)^2$ with $r\mid c$. Hence, from \eqref{Eq:fg}, we see that
 \[g(x)=x^3+\left(r-(c/r)^2\right)x+c=\left(x+c/r\right)\left(x^2-(c/r)x+r\right),\] and therefore,
 \[f(x)=\left(x^2+c/r\right)\left(x^4-(c/r)x^2+r\right),\]
 contradicting the fact that $f(x)$ is irreducible over $\Q$. Consequently, $\Gal(f)\not \simeq C_6$ in this case as well.

 Finally, suppose that $a\ne 0$ and $b=0$, so that
  \begin{equation}\label{Eq:fg2}
  f(x)=x^6+ax^4+c \quad \mbox{and}\quad g(x)=x^3+ax^2+c.
   \end{equation} Assume, by way of contradiction, that $\Gal(f)\simeq C_6$. Then, by item \eqref{C6} of Lemma \ref{Lem:AJ}, we have that
   $h(x)=x^6+acx^2-c^2$ is reducible over $\Q$. Define
  \begin{equation}\label{Eq:h2}
  h_2(x):=x^3+acx-c^2.
   \end{equation}
   If $h_2(x)$ is irreducible over $\Q$, then since $h(x)$ is reducible, it follows from \eqref{Eq:HJMAconditions} of Lemma \ref{Lem:HJMS}, with \[G(x)=h_2(x),\quad F(x)=h(x),\quad A=0, \quad B=ac\quad  \mbox{and} \quad C=c\] that
   \begin{equation}\label{Eq:HJMSconditions2}
  2n=m^2 \quad \mbox{and} \quad ac=n^2-2mc.
  \end{equation}
  Combining the two equations in \eqref{Eq:HJMSconditions2}, we get that
  \begin{equation}\label{Eq:b2}
  a=\frac{m^4-8mc}{4c}.
  \end{equation} Then, by \eqref{Eq:b2}, we have that
  \[\Delta(g)=-c(4a^3+27c)=\frac{-(m^6-20m^3c+108c^2)(m^3-2c)^2}{16c^2}.\]
   Observe that $2\mid m$ from \eqref{Eq:HJMSconditions2}. Hence,
  \[-(m^6-20m^3c+108c^2)\equiv 20c^2 \pmod{32},\] which implies, as in the previous case, that $\Delta(g)$ is not a square, contradicting item \eqref{C6} of Lemma \ref{Lem:AJ}.
  Therefore, $h_2(x)$ must be reducible. Let $r\in \Z$ be a zero of $h_2(x)$. Then $r^3+acr-c^2=0$ from \eqref{Eq:h2}, so that $a=(c^2-r^3)/(cr)$ with $r\mid c$. Thus, from \eqref{Eq:fg2}, we see that
 \[g(x)=x^3+\left(\frac{c^2-r^3}{cr}\right)x^2+c=\left(x+c/r\right)\left(x^2-(r^2/c)x+r\right),\] and therefore,
 \[f(x)=\left(x^2+c/r\right)\left(x^4-(r^2/c)x^2+r\right),\]
 contradicting the fact that $f(x)$ is irreducible over $\Q$. Thus, $\Gal(f)\not \simeq C_6$ in this final case, and consequently, there exist no monogenic even cyclic sextic binomials or trinomials, which completes the proof of item \eqref{Main I:1}.

 For item \eqref{Main I:2}, we first note, in the context of Lemma \ref{Lem:AJ}, that $g(x)=x(x+a)^2+b$ and $f(x)=g(x^2)$, with
 \begin{equation}\label{Eq:Delgf2}
 \Delta(g)=b(4a^3-27b) \quad \mbox{and} \quad \Delta(f)=-2^6b^3(4a^3-27b)^2, 
 \end{equation}
  by \cite[Theorem 2]{HJDisc} (or simply Maple). For brevity of notation, we define
 \begin{equation}\label{Eq:delta}
 \delta:=4a^3-27b.
 \end{equation} If $\Gal(f)\simeq C_6$, then $\Delta(g)$ is a square by item \eqref{C6} of Lemma \ref{Lem:AJ}. Conversely, suppose then that $\Delta(g)$ is a square. If $-b$ is a square, say $-b=t^2$, then
 \[f(x)=x^2(x^2+a)^2-t^2=\left(x(x^2+a)-t\right)\left(x(x^2+a)+t\right),\]
 contradicting the fact that $f(x)$ is irreducible. Hence, $-b$ is not a square. With $h(x)$ as defined in Lemma \ref{Lem:AJ}, we have here that
 \[h(x)=x^6-a^2x^4+2abx^2-b^2=(x^3-ax^2+b)(x^3+ax^2-b).\] Thus, $\Gal(f)\simeq C_6$ by item \eqref{C6} of Lemma \ref{Lem:AJ}.

 Let $\FF_1$ be as defined in \eqref{F1}. From the previous argument, every $f(x)\in \FF_1$ is such that $\Gal(f)\simeq C_6$. To see that $\FF_1$ is infinite, let $n\ge 1$ be an integer, let $a=9n^2+3n+7$ and let $b=a^2$. Then
 \[f(x)=x^2(x^2+9n^2+3n+7)^2+(9n^2+3n+7)^2\equiv x^6+2x^4+x^2+1 \pmod{3},\] and since $x^6+2x^4+x^2+1$ is irreducible in $\F_3[x]$, we have that $f(x)$ is irreducible over $\Q$ for all integers $n\ge 1$. From \eqref{Eq:Delgf2}, an easy computation reveals that
 \[\Delta(g)=(6n+1)^2(9n^2+3n+7)^4\] is a square,
 which implies that $f(x)\in \FF_1$. Hence, $\FF_1$ is infinite.

 To determine the monogenic polynomials in $\FF_1$, we assume that
 \[f(x)=x^2(x^2+a)+b=x^6+2ax^4+a^2+b\in \FF_1\] is monogenic. Then, for primes dividing $\Delta(f)$ in \eqref{Eq:Delgf2}, we use the fact that $\Delta(g)$ is a square and Theorem \ref{Thm:JLY} to derive necessary criteria on the coefficients of $f(x)$.

  It is easy to see that condition \eqref{JLY:1} of Theorem \ref{Thm:JLY} implies that $b$ is squarefree. We also see from \eqref{Eq:Delgf2} that condition \eqref{JLY:2} of Theorem \ref{Thm:JLY} is applicable only for primes $p\in \{2,3\}$. 
      Clearly, $p=2$ divides $\Delta(f)$ from \eqref{Eq:Delgf2}. Since $f(x)\in \FF_1$, we know that
 $\Delta(g)$ is a square, and therefore, from \eqref{Eq:Delgf2}, we have that $b\mid 2a$. Thus, $\delta/b\in \Z$, where $\delta$ is as in \eqref{Eq:delta},  and we see from \eqref{Eq:Delgf2} that $\Delta(g)=b^2(\delta/b)$ so that $\delta/b$ is a square. It is easy to verify that
 \[\delta/b\equiv \left\{ \begin{array}{cl}
   5 \pmod{8} & \mbox{if $2\mid a$},\\[4pt]
   3 \pmod{4} & \mbox{if $2\nmid a$ and $2\mid b$,}
 \end{array}\right.\] which in these cases contradicts the fact that $\delta/b$ is a square. Thus, we conclude that $2\nmid ab$. One consequence is that $b\mid a$. Additionally, it  follows from condition \eqref{JLY:3} of Theorem \ref{Thm:JLY} that $b\equiv 1 \pmod{4}$.

 Suppose that $p\mid \Delta(f)$ and $p\nmid 2ab$. Then $p\mid \delta$ from \eqref{Eq:Delgf2}, which implies that $p^2\mid \delta/b$ from \eqref{Eq:Delgf2} since $\delta/b$ is a square. However, since $f(x)$ is monogenic, we have that $p^2\nmid \delta$ by condition \eqref{JLY:4} of Theorem \ref{Thm:JLY}, which yields the contradiction that $p^2\nmid \delta/b$. Consequently, every prime dividing $\Delta(f)$ must divide $2ab$, or more simply $2a$, since $b\mid a$.

We consider next the prime $p=3$, and we suppose that $3\mid \Delta(f)$. Recall that $3\mid a$ by the previous discussion. Then $\delta/(9b)\in \Z$ from \eqref{Eq:Delgf2}, and $\delta/(9b)$ is a square since $\delta/b$ is a square. If $9\mid a$, then it is easy to confirm that
\[\delta/(9b)\equiv 6 \pmod{9},\] which contradicts the fact that $\delta/(9b)$ is a square. Hence, $3\mid \mid a$. Then
\begin{equation}\label{Eq:delta/9b}
\delta/(9b)= \left\{ \begin{array}{cl}
   4(a/b)(a/3)^2-3 & \mbox{if $3\mid b$},\\[4pt]
   3\left(4(a/3b)(a/3)^2-1\right) & \mbox{if $3\nmid b$.}
 \end{array}\right.
 \end{equation}
Let $p$ be a prime divisor of $\delta/(9b)$. 

We first consider the case when $3\mid b$ in \eqref{Eq:delta/9b}. We know from previous arguments that $p\mid 2a$. Clearly, $p\ne 2$. Since $3\mid \mid a$, we see conclude that $3\nmid (a/b)(a/3)$, so that $p\ne 3$. If $p>3$, then $p\mid (a/b)(a/3)$, which implies the contradiction that $p\mid 3$. Thus, since $\delta/(9b)$ is a square, we must have $\delta/(9b)=1$, which implies that
 \[(a/b)(a/3)^2=1.\] Hence, since $b\equiv 1 \pmod{4}$, we conclude that $a=b=-3$, which yields the single polynomial $f(x)=x^6-6x^4+9x^2-3$. It is easily confirmed that $f(x)$ is monogenic with $\Gal(f)\simeq C_6$.

We consider next the case when $3\nmid b$ in \eqref{Eq:delta/9b}. If $p>3$, then since $p\mid a$, we have that $p\mid (a/3)$ which yields the contradiction $p\mid 1$. Hence, $p=3$ is the only prime divisor of the square $\delta/(9b)$, and so we have
\begin{equation}\label{Eq:Dio1}
4(a/3b)(a/3)^2-1=3^{2k+1},
\end{equation}
for some nonnegative integer $k$. Suppose that $k\ge 1$. Then we have from \eqref{Eq:Dio1} that
\begin{equation}\label{Eq:Congb}
b\equiv 4(a/3)^3\equiv \pm 4 \pmod{9}.
\end{equation} We give details only in the case $b\equiv 4 \pmod{9}$ since the case $b\equiv -4 \pmod{9}$ is similar.
We examine condition \eqref{JLY:2} of Theorem \ref{Thm:JLY} with $p=3$ since $3\mid a$ and $3\nmid b$. From condition \eqref{JLY:2} and \eqref{Eq:Congb}, we have
\[b_1^3=(2a/3)^3\equiv 2(4(a/3)^3)\equiv -1 \Mod{9} \quad \mbox{and} \quad c_1^3=\left(\frac{b-b^3}{3}\right)^3\equiv 1 \Mod{9}.\] Since $b^2\equiv 7\pmod{9}$,
 it follows that
 \[b(b_1^3b^2+c_1^3)\equiv 4((-1)(7)+1)\equiv 3 \pmod{9},\]
 which implies that condition \eqref{JLY:2} of Theorem \ref{Thm:JLY} is not satisfied, and consequently, $f(x)$ is not monogenic. That leaves the case $k=0$
 in \eqref{Eq:Dio1}. In that situation, we see that
 \[(a/3b)(a/3)^2=1,\] from which we deduce that $a=3$ and $b=1$, since $b\equiv 1 \pmod{4}$. Hence, we get the single polynomial
 $f(x)=x^6+6x^4+9x^2+1$, which is easily verified to be monogenic with $\Gal(f)\simeq C_6$.

  We turn now to the proof of item \eqref{Main I:3}. Although the approach is similar to the proof of item \eqref{Main I:2}, we provide details for the sake of completeness. We first note, in the context of Lemma \ref{Lem:AJ}, that $g(x)=x^3+a(bx+1)^2$ with
 \begin{equation}\label{Eq:Delgf3}
 \Delta(g)=a^2(4ab^3-27)\quad \mbox{and} \quad \Delta(f)=-2^6a^5(4ab^3-27)^2,
  \end{equation}
  by \cite[Theorem 1]{HJDisc} (or simply Maple). For brevity of notation, we define
 \begin{equation}\label{Eq:delta3}
 \delta:=4ab^3-27.
 \end{equation} If $\Gal(f)\simeq C_6$, then $\Delta(g)$ is a square by item \eqref{C6} of Lemma \ref{Lem:AJ}, and therefore from \eqref{Eq:Delgf3}, we see that $\delta$ is a square. Conversely, suppose that $\delta$ is a square, so that $\Delta(g)$ is a square from \eqref{Eq:Delgf3}. If $-a$ is a square, say $-a=t^2$, then
 \[f(x)=(x^3-btx^2-t)(x^3+btx^2+t),\]
 which contradicts the fact that $f(x)$ is irreducible over $\Q$. Hence, $-a$ is not a square. With $h(x)$ as defined in Lemma \ref{Lem:AJ}, we see that
 \[h(x)=x^6-2abx^4+a^2b^2x^2-a^2=(x^3-abx-a)(x^3-abx+a).\] Thus, $\Gal(f)\simeq C_6$ by item \eqref{C6} of Lemma \ref{Lem:AJ}.

 Let $\FF_2$ be as defined in \eqref{F2}. From the argument above, every $f(x)\in \FF_2$ is such that $\Gal(f)\simeq C_6$.
 To see that $\FF_2$ is infinite, let $a=9n^2+15n+13$ and let $b=1$, where $n\in \Z$. Then
 \[f(x)=x^6+(9n^2+15n+13)(x^2+1)^2\equiv x^6+x^4+2x^2+1 \pmod{3},\]
  and since $x^6+x^4+2x^2+1$ is irreducible in $\F_3[x]$, we conclude that $f(x)$ is irreducible over $\Q$ for all integers $n$. From \eqref{Eq:Delgf3}, an easy computation confirms that
 \[\Delta(g)=(6n+5)^2(9n^2+15n+13)^2\] is a square,
 which implies that $f(x)\in \FF_2$. Hence, $\FF_2$ is infinite.

  To determine the monogenic polynomials in $\FF_2$, we assume that
 \[f(x)=x^6+a(bx^2+1)^2=x^6+ab^2x^4+2abx^2+a\in \FF_2\] is monogenic. Then, for primes dividing $\Delta(f)$ in \eqref{Eq:Delgf3}, we use Theorem \ref{Thm:JKK}, and the fact that $\Delta(g)$ is a square, to derive necessary criteria on the coefficients of $f(x)$. 

We first observe that condition \eqref{JKK:1} of Theorem \ref{Thm:JKK} implies that $a$ is squarefree, while condition \eqref{JKK:2} is only applicable when $p\in \{2,3\}$.
Since $f(x)\in \FF_2$, it follows that
 $\Delta(g)$ is a square, and therefore, from \eqref{Eq:Delgf3}, we see that $\delta$ is a square. An easy calculation shows that if $2\mid ab$, then $\delta\equiv 5\pmod{8}$, which is impossible since 5 is not a square modulo 8. Hence, we conclude that $2\nmid ab$ and that condition \eqref{JKK:2} of Theorem \ref{Thm:JKK} is not applicable for $p=2$. Consequently, it follows from condition \eqref{JKK:3} of Theorem \ref{Thm:JKK} that $a\equiv 1 \pmod{4}$.

Suppose that $p\mid \Delta(f)$ and $p\nmid 2ab$. Then $p\mid \delta$ from \eqref{Eq:Delgf3}, which implies that $p^2\mid \delta$ from \eqref{Eq:Delgf3} since $\delta$ is a square. However, $p^2\nmid \delta$ by condition \eqref{JLY:4} of Theorem \ref{Thm:JLY} since $f(x)$ is monogenic.  Thus, every prime dividing $\Delta(f)$ must divide $2ab$. Consequently, if $p\mid \delta$, then $p=3$ and
\begin{equation}\label{Eq:deltaDio}
\delta=4ab^3-27=3^{2k},
\end{equation}
for some integer $k\ge 0$.

If $k=0$ in \eqref{Eq:deltaDio}, then $a=-7$ since $a\equiv 1 \pmod{4}$, and $b=-1$. Hence, in this case, we get the quadrinomial $x^6-7x^4+14x^2-7$, which is easily confirmed to be monogenic and cyclic.

Suppose then that $k\ge 1$ in \eqref{Eq:deltaDio}, so that $3\mid \delta$. If $3\mid a$, then we see from \eqref{Eq:deltaDio} that $3\mid b$ since $a$ is squarefree. Therefore, we have from \eqref{Eq:deltaDio} that
\[3^{2k-3}=\delta/3^3=4a(b/3)^3-1\equiv 2 \pmod {3},\]
which contradicts the fact that $k\ge 1$. Hence, $3\nmid a$, so that $3\mid b$, since every prime dividing $\delta$ must divide $2ab$. A similar argument shows that $3\mid \mid b$.
Therefore, $k\ge 2$ and
\begin{equation}\label{Eq:delta/27}
\delta/3^3=4a(b/3)^3-1=3^{2k-3}.
\end{equation} Suppose that $k=2$. Then, since $a\equiv 1 \pmod{4}$, the only solution to \eqref{Eq:delta/27} is $a=1$ and $b=3$, producing the quadrinomial $x^6+9x^4+6x^2+1$, which is easily verified to be monogenic and cyclic.

Suppose then that $k\ge 3$ in \eqref{Eq:delta/27}. Then, since $3\nmid a$ and $3\mid \mid b$, an easy computation shows that the only solutions to the congruence $4a(b/3)^3\equiv 1 \pmod{27}$ are
\begin{equation}\label{Eq:mod27}
(a\mmod{27},b\mmod{27})\in \{(2,15), (7,3), (11,6), (16,21),  (20,24), (25,12)\}.
\end{equation} With $p=3$,
\[b_1=2ab/3 \quad \mbox{and} \quad c_1=(a-a^3)/3\] as in condition \eqref{JKK:2} of Theorem \ref{Thm:JKK}, we have that $3\nmid b_1$ and a straightforward calculation reveals that $-c_1^3+ab_1^3\equiv 0 \pmod{3}$ for all possibilities in \eqref{Eq:mod27}. Thus, no additional monogenic polynomials arise from solutions to \eqref{Eq:delta/27} when $k\ge 3$, and the proof of item \eqref{Main I:3} is complete.

  For item \eqref{Main I:4}, we first note that $\FF_3$ is an infinite set. We define
  \[g_n(x):=x^3+(n^2+5)x^2+(n^2+2n+6)x+1,\] and we use Maple to calculate
  \begin{align}\label{Eq:Delgf3}
  \begin{split}
  \Delta(g_n)&=(n^2+n-1)^2(n^2+n+7)^2 \quad \mbox{and}\\
  \Delta(f_n)&=-2^6(n^2+n-1)^4(n^2+n+7)^4.
  \end{split}
  \end{align}
    Observe that $2\nmid \Delta(g_n)$.

We first show that $f_n(x)$ is irreducible over $\Q$. Since the equation
\[f_n(\pm 1)=2n^2+2n+13=0,\] has no integer solutions, it follows by the Rational Root theorem that $f_n(x)$ has no linear factors. Suppose that $f_n(x)$ has an irreducible quadratic factor $s(x)$. Then the remainder of $f_n(x)$ divided by $s(x)$ must be zero. For $s(x)=x^2+ax+1$, we use Maple to carry out this polynomial long division to find the remainder $R(x)=Ax+B$. Setting $A=0$ and $B=0$, and solving this system of equations using Maple, we get two solutions, both of which have $n=-1/2$, which is impossible. Similar obvious contradictions arise if $s(x)=x^2+ax-1$. Thus, $f_n(x)$ has no irreducible quadratic factor. Suppose then that $s(x)=x^3+ax^2+bx\pm 1$ is an irreducible factor of $f_n(x)$. In this case, $R(x)=Ax^2+Bx+C$, and we solve the system $\{A=0,B=0,C=0\}$ in Maple and get two solutions. In the first solution we have that $a$ must be a root of $w_1(z):=z^4+2z^3+3z^2-2z+1$, while in the second solution, $a$ must be a root of $w_2(z):=z^4-2z^3+3z^2+2z+1$. However, since $w_i(\pm 1)=5$, $w_i(z)$ has no integer roots by the Rational Root theorem. Thus, we conclude that $f_n(x)$ is irreducible over $\Q$.

Next, in the context of Lemma \ref{Lem:AJ}, with
\[a:=n^2+5, \ \ b:=n^2+2n+6, \ \ c:=1, \ \ g(x):=g_n(x) \quad \mbox{and} \quad f(x):=f_n(x),\] we have $-c=-1$ is not a square, $\Delta(g)$ is clearly a square by \eqref{Eq:Delgf3} and
\begin{align*}
h(x)&=x^6-(n^2+2n+6)x^4+(n^2+5)x^2-1\\
&=(x^3+(n+2)x^2+(n-1)x-1)(x^3-(n+2)x^2+(n-1)x+1).
\end{align*} Hence, it follows from item \eqref{C6} of Lemma \ref{Lem:AJ} that $\Gal(f_n)\simeq C_6$.

To address the monogenicity of $f_n(x)$, we note that, since $f_n(x)$ is irreducible and $f_n(x)=g_n(x^2)$, then $g_n(x)$ is irreducible and the monogenicity of $g_n(x)$ is a necessary condition for the monogenicity of $f_n(x)$. Consequently, we focus first on determining the values of $n$ for which $g_n(x)$ is not  monogenic. To accomplish this task, we let $p$ be a prime divisor of $\Delta(g_n)$ and we use Theorem \ref{Thm:Dedekind} with $T(x):=g_n(x)$. 
Let $K=\Q(\theta)$, where $T(\theta)=0$, and let $\Z_K$ denote the ring of integers of $T(x)$.

We show first that if $n\not \in N:=\{-2,-1,0,1\}$, then $g_n(x)$ is not monogenic. 
Note that $\abs{n^2+n-1}=1$ if and only if $n\in N$.
Suppose then that $n\not \in N$, and let $p$ be a prime divisor of $n^2+n-1$. Then, since $n^2\equiv -n+1 \pmod{p}$, straightforward calculations reveal that
\[\overline{T}(x)=x^3-(n-6)x^2+(n+7)x+1=(x-(3n-2))(x+(n+2))^3.\] Hence, we can let
\[h_1(x)=(x-(3n-2))(x+(n+2)) \quad \mbox{and} \quad h_2(x)=(x+(n+2))^2\]
in Theorem \ref{Thm:Dedekind}. With $F(x)$ as defined in Theorem \ref{Thm:Dedekind}, an easy calculation yields
\begin{align*}
F(x)&=\frac{(x-(3n-2))(x+(n+2))^3-(x^3+(n^2+5)x^2+(n^2+2x+6)x+1)}{p}\\
&=-\left(\frac{n^2+n-1}{p}\right)x^2-6\left(\frac{n^2+n-1}{p}\right)x-\frac{(3n+7)(n^2+n-1)}{p}.
\end{align*} Since $\gcd(\overline{h_1},\overline{h_2})=x+(n+2)$, to determine whether $\gcd(\overline{F},\overline{h_1},\overline{h_2})=1$, we only need to calculate $\overline{F}(-(n+2))$. Since
\[F(-(n+2))=-\frac{(n^2+n-1)^2}{p}\equiv 0 \pmod{p},\] we conclude that $\gcd(\overline{F},\overline{h_1},\overline{h_2})\ne 1$ and therefore, $g_n(x)$ is not monogenic. Consequently, it follows that $f_n(x)$ is not monogenic when $n\not \in N$.

When $n\in N$, we have that
 \begin{align}\label{Eq:fn}
 \begin{split}
  f_{-2}(x)&=x^6+9x^4+6x^2+1 \quad \mbox{with} \quad \Delta(f_{-2})=-2^63^8\\
  f_{-1}(x)&=x^6+6x^4+5x^2+1 \quad \mbox{with} \quad \Delta(f_{-1})=-2^67^4\\
  f_{0}(x)&=x^6+5x^4+6x^2+1 \quad \mbox{with} \quad \Delta(f_0)=-2^67^4\\
  f_1(x)&=x^6+6x^4+9x^2+1 \quad \mbox{with} \quad \Delta(f_1)=-2^63^8.
  \end{split}
 \end{align}
The quadrinomials $f_{-2}(x)$ and $f_{1}(x)$ in \eqref{Eq:fn} have been previously identified as monogenic in, respectively, $\QQ_1$ and $\QQ_2$. The quadrinomials $f_{-1}(x)$ and $f_0(x)$ in \eqref{Eq:fn} are new, and using Theorem \ref{Thm:Dedekind} (or simply Maple), it is easy to confirm that they are monogenic, which completes the proof of item \eqref{Main I:4}.

  Finally, for item \eqref{Main I:5}, we have that
  \[\QQ_1\cup \QQ_2\cup \QQ_3=\{P_1(x),P_2(x),P_3(x),P_4(x),P_5(x),P_6(x)\},\] where
 \begin{align}\label{Eq:4quads}
 \begin{split}
  P_1(x):&=x^6-6x^4+9x^2-3 \quad \mbox{with} \quad \Delta(P_1)=2^63^9,\\
  P_2(x):&=x^6+6x^4+9x^2+1 \quad \mbox{with} \quad \Delta(P_2)=-2^63^8,\\
  P_3(x):&=x^6-7x^4+14x^2-7 \quad \mbox{with} \quad \Delta(P_3)=-2^67^5,\\
  P_4(x):&=x^6+9x^4+6x^2+1 \quad \mbox{with} \quad \Delta(P_4)=-2^63^8,\\
  P_5(x):&=x^6+5x^4+6x^2+1 \quad \mbox{with} \quad \Delta(P_4)=-2^67^4,\\
  P_6(x):&=x^6+6x^4+5x^2+1 \quad \mbox{with} \quad \Delta(P_4)=-2^67^4.
  \end{split}
 \end{align}
   Since the quadrinomials $P_i(x)$ in \eqref{Eq:4quads} are all monogenic, it follows from Theorem \ref{Thm:mono} that they are distinct if and only if their polynomial discriminants are not equal. Item \eqref{Main I:5} then follows by inspection of \eqref{Eq:4quads}, which completes the proof of the theorem.
 \end{proof}
\begin{rem}
  Computer evidence suggests that the four quadrinomials given in item \eqref{Main I:5} of Theorem \ref{Thm:Main} are the only distinct monogenic even cyclic sextic polynomials.
\end{rem}

\section{Some noncyclic monogenic even sextic quadrinomials}\label{Sec:noncyclic}
In this section, we show that the situation when the Galois group of an even sextic quadrinomial is not $C_6$ can be quite different from Theorem \ref{Thm:Main}. 
In particular, we prove the following:
\begin{thm}\label{Thm:A4}
  For any $n\in \Z$, define
\[f_n(x):=x^6+(3n+4)x^4+(3n+1)x^2-1\quad \mbox{and}\quad \D:=9n^2+15n+13.\] Then
\begin{enumerate}
  \item \label{A4 I:1} $f_n(x)$ is irreducible over $\Q$ and $\Gal(f_n)\simeq A_4$,
  \item \label{A4 I:2} $f_n(x)$ is monogenic if and only if $\D$ is squarefree,
  \item \label{A4 I:3} the set
  \[\FF:=\{f_n(x): \mbox{$\D$ is squarefree}\}.\] is an infinite collection of distinct monogenic $A_4$ even sextic quadrinomials. Here the term ``distinct" means that no two quadrinomials in $\FF$ generate the same sextic field.
\end{enumerate}
\end{thm}

\begin{proof}
  We begin by proving that $f_n(x)$ is irreducible over $\Q$. Since
  \[f_n(\pm1)=6n+5\ne 0,\] we have by the Rational Root theorem that $f_n(x)$ has no linear factors. If $f_n(x)$ has an irreducible quadratic factor $s(x)$, then $s(x)=x^2+ax \pm 1$. Suppose that $s(x)=x^2+ax+1$ is such a factor of $f_n(x)$. Then, using Maple to perform the long polynomial division of $f_n(x)$ divided by $s(x)$, we see that the remainder is
  \[R(x)=(3an+4a-3a^3n-a^5)x+1-a^4-a^2-3a^2n.\] Since $R(x)=0$, we use Maple to solve the system of equations
  \[\{3an+4a-3a^3n-a^5=0,\ 1-a^4-a^2-3a^2n=0\}. \] The solution that Maple gives has $n=-5/6$, which is impossible. The same contradiction is reached if $s(x)=x^2+ax-1$. Suppose then that $f_n(x)$ has an irreducible cubic factor $s(x)=x^3+ax^2+bx+1$. Proceeding as before, we use Maple to solve a system of three equations generated by setting the coefficients of the remainder equal to zero. Maple gives two solutions. In the first solution, we have that $b$ is a root of the polynomial $u(z)=z^2-2z+3$. However, $\Delta(u)=-8$ which implies that $u(z)$ has no integer roots. In the second solution provided by Maple, we have that $a$ is a root of $w(z)=z^2-2z+b^2-2b+3$. Since $\Delta(w)=-4((b-1)^2+1)<0$, it follows that $w(z)$ has no integer roots. Hence, this scenario is impossible. The same contradictions arise if $s(x)=x^3+ax^2+bx-1$. Thus, we conclude that $f_n(x)$ is irreducible over $\Q$.

  We note, in the context of Lemma \ref{Lem:AJ}, that
  \[a:=3n+4, \quad b:=3n+1 \quad \mbox{and} \quad c:=-1,\] so that
  \[g(x):=x^3+(3n+4)x^2+(3n+1)x-1 \quad \mbox{and} \quad h(x):=x^6-(3n+1)x^4-(3n+4)x^2-1.\] Clearly, $-c=1$ and $\Delta(g)=\D^2$ are squares. Since the same techniques used to show that $f_n(x)$ is irreducible can be used to confirm that $h(x)$ is irreducible over $\Q$, we omit the details. Thus, it follows from item \eqref{A4} of Lemma \ref{Lem:AJ} that $\Gal(f_n)\simeq A_4$, which completes the proof of item \eqref{A4 I:1}.

  For item \eqref{A4 I:2}, an easy calculation in Maple gives $\Delta(f_n)=2^6\D^4$. Let $K=\Q(\theta)$, where $f_n(\theta)=0$, and let $\Z_K$ denote the ring of integers of $K$. We let $p$ be a prime divisor of $\Delta(f_n)$, and we use Theorem \ref{Thm:Dedekind} with $T(x):=f_n(x)$ to determine exactly when $p$ does, or does not, divide the index $[\Z_K:\Z[\theta]]$, which will provide necessary and sufficient conditions for the monogenicity of $f_n(x)$.

  Suppose first that $p=2$. Then
  \[\overline{T}(x)=\left\{\begin{array}{cl}
    (x^3+x+1)^2 & \mbox{if $2\mid n$}\\
    (x^3+x^2+1)^2 & \mbox{if $2\nmid n$}.
  \end{array}\right.\]
   When $2\mid n$, we can let $h_1(x)=h_2(x)=x^3+x+1$ in Theorem \ref{Thm:Dedekind}. Thus,
  \begin{align*}
  F(x)&=-\left(\frac{3n+2}{2}\right)x^4+x^3-\left(\frac{3n}{2}\right)x^2+x+1\\
  &\equiv \left\{\begin{array}{cl}
    (x+1)^2(x^2+x+1) \pmod{2} & \mbox{if $n\equiv 0 \pmod{4}$}\\
    (x+1)^3 \pmod{2} & \mbox{if $n\equiv 2 \pmod{4}$}.
  \end{array}\right.
  \end{align*} Hence, $\gcd(\overline{F},\overline{h_1},\overline{h_2})=1$ in this case.
  When $2\nmid n$, we can let $h_1(x)=h_2(x)=x^3+x^2+1$ in Theorem \ref{Thm:Dedekind}. Thus,
  \begin{align*}
  F(x)&=x^5-\left(\frac{3n+3}{2}\right)x^4+x^3-\left(\frac{3n-1}{2}\right)x^2+1\\
  &\equiv \left\{\begin{array}{cl}
    x^5+x^4+x^3+x^2+1 \pmod{2} & \mbox{if $n\equiv 1 \pmod{4}$}\\
    x^5+x^3+1 \pmod{2} & \mbox{if $n\equiv 3 \pmod{4}$}.
  \end{array}\right.
  \end{align*}
   Hence, $\gcd(\overline{F},\overline{h_1},\overline{h_2})=1$ in this case as well, and we conclude that $2\nmid [\Z_K:\Z[\theta]]$.

  Now suppose that $p\mid \D$. Then $p\not \in \{2,3\}$ and since $9n^2+15n+13\equiv 0 \pmod{p}$, we have
  \[n\equiv \frac{-5\pm 3\sqrt{-3}}{6} \pmod{p}.\] Suppose that $n\equiv (-5+3\sqrt{-3})/6 \pmod{p}$. Then $\sqrt{-3}\equiv (6n+5)/3 \not \equiv 0 \pmod{p}$ and
  \begin{align*}
  \overline{T}(x)&=\left(x-\frac{-1+\sqrt{-3}}{2}\right)^3\left(x+\frac{-1+\sqrt{-3}}{2}\right)^3\\
  &=(x-r)^3(x+r)^3,
    \end{align*}  where $r\in \Z$ with $r\equiv (3n+1)/3 \pmod{p}$. Thus, we can let
    \[h_1(x)=(x-r)(x+r) \quad \mbox{and}\quad  h_2(x)=(x-r)^2(x+r)^2\]
  in Theorem \ref{Thm:Dedekind}. Then
  \begin{multline*}
    \overline{F}(x)=-\overline{\left(\frac{\D}{3p}\right)}x^4+\overline{\left(\frac{(3n-2)(3n+1)\D}{3^3p}\right)}x^2\\
    -\overline{\left(\frac{(3n-2)(3n+4)(9n^2-3n+7)\D}{3^6p}\right)}
  \end{multline*}
  so that
  \[\overline{F}(r)=\overline{F}((3n+1)/3)=-\overline{\left(\frac{\D}{p}\right)\left(\frac{6n+5}{3^3}\right)}.\] Since $6n+5\not \equiv 0 \pmod{p}$, it follows that $f_n(x)$ is monogenic if and only if $p^2$ does not divide $\D$ for all prime divisors $p$ of $\D$, which is true if and only if $\D$ is suqarefree. Since the argument is the same if $n\equiv (-5-3\sqrt{-3})/6 \pmod{p}$, we omit the details in that case, and the proof of item \eqref{A4 I:2} is complete.

Finally, for item \eqref{A4 I:3}, since from \cite{Nagel}, we have that $9n^2+15n+13$ is squarefree for infinitely many integers $n$, we conclude that the set  $\FF$ is infinite. To see that all quadrinomials in $\FF$ are distinct, we assume, by way of contradiction, that $f_{n_1}(x), f_{n_2}(x)\in \FF$ are not distinct, where $n_1\ne n_2$. Thus, $K_1=K_2$, where $K_i=\Q(\theta_i)$ and $f_{n_i}(\theta_i)=0$.
 Since each $f_{n_i}(x)$ is monogenic, it follows from Theorem \ref{Thm:mono} that $\Delta(f_{n_1})=\Delta(f_{n_2})$. Thus, we arrive at the possible equations
\begin{equation}\label{Eq:Equaldisc}
9n_1^2+15n_1+13=\pm(9n_2^2+15n_2+13).
\end{equation}
The plus sign in \eqref{Eq:Equaldisc} yields the contradiction $n_1+n_2=-5/3$ if $n_1\ne n_2$, while the negative sign requires that
$-36(n_2+5/6)^2-54$ be a square, which is also impossible. Hence, the proof of the theorem is complete.
 \end{proof}

\begin{rems}
  The fact that $\Gal(f_n)\simeq A_4$ in item \eqref{A4 I:1} of Theorem \ref{Thm:A4} also follows from \cite[Theorem 1.4]{IJ} and \cite[Proposition 4.4]{HJMS}. As far as the author can determine, the explicit family $\FF$ in Theorem \ref{Thm:A4} is new and does not appear in the current literature.
\end{rems}







\end{document}